\documentclass[11pt]{amsart}
\usepackage{amsmath}
\usepackage{hyperref}
\newtheorem{proposition}{Proposition}
\newtheorem{lemma}[proposition]{Lemma}

\theoremstyle{definition}

\newtheorem{example}[proposition]{Example}

\theoremstyle{remark}

\theoremstyle{acknowledgment}

\theoremstyle{formula}

\newcommand\inner[2]{\langle #1, #2 \rangle}

\makeatletter
\newcommand{\addresseshere}{%
	\enddoc@text\let\enddoc@text\relax
}
\makeatother

\begin{document}
	
\title{A Derivation of Classical Orthogonal Polynomials using Generalized Vandermonde Determinants}
\author{Lijing Wang}
\address{Global Model Risk Management, Bank of America }
\email{wlj@alum.mit.edu}

\subjclass{Primary 15B05, 42C05; Secondary 15A63, 15A75}
\keywords{Hermite polynomials, Jacobi polynomials, Laguerre polynomials, Gram-Schmidt orthogonalization, Vandermonde determinants, Hodge star operator}

\begin{abstract}
We present a derivation of classical Hermite, Laguerre, and Jacobi orthogonal polynomials directly through the Gram-Schmidt orthogonization process. The derivation uses certain generalized Vandermonde determinants with entries defined by Gamma and Beta functions. We also provide a geometric formulation of Gram-Schmidt orthogonalization using the Hodge star operator.
\end{abstract}

\maketitle

\section{Introduction}

Orthogonal polynomials are classes of polynomials satisfying orthogonal relationships with respect to certain weighting functions (\cite{kn:AW},\cite{kn:BealsWong},\cite{kn:CH}). Three well-known classical polynomials are Hermite, Laguerre, and Jacobi polynomials. There are various well-known ways to derive the analytic form of these polynomials such as using Sturm-Liouville theory, three-terms recursive formula or Rodrigues formulas. Gram-Schmidt orthogonalization is also often used to derive first few terms but not general terms. 

In this note, we present a derivation of the general terms of the three classical polynomials using Gram-Schmidt orthogonalization process. A key ingredient of the approach is the calculation of certain generalized Vandermonde determinants with entries defined by Gamma and Beta functions. Such generalized determinants have been studied by Normand in \cite{kn:Normand}.

Given an interval $I=(a, b)$ and a weight function $w(x)$, one can define a set of orthogonal polynomials $\{P_n(x)\}_{n=0, 1, 2,\cdots}$ so that they form an orthogonal basis with respect to the inner product defined as
\[ \inner{f}{g}:=\int_{a}^{b} f(x)g(x)w(x)dx \]

The three classical polynomials correspond to the following cases:
\begin{itemize}
	\item Hermite polynomials $H_{n}(x)$: $I=(-\infty, +\infty)$ and $w(x)=e^{-x^{2}}$
	\item Laguerre polynomials $L_{n}^{\alpha}(x)$: $I=(0, +\infty)$ and $w(x)=x^{\alpha}e^{-x}$
	\item Jacobi polynomials $J_{n}^{(\alpha, \beta)}$: $I=(-1, +1)$ and $w(x)=(1-x)^{\alpha}(1+x)^{\beta}$
\end{itemize}

Jacobi polynomials also have Gegenbauer polynomials, Chebyshev polynomials and Legendre polynomials as special cases.

Explicit expressions for the three classical polynomials (\cite{kn:BealsWong}) are given as
\begin{align}
	H_{n}(x) &=\sum_{m=0}^{\lfloor \frac{n}{2} \rfloor}\frac{(-1)^{m}}{(m)!(n-2m)!}(2x)^{n-2m} \label{eq:Hermite}\\
	L_{n}^{\alpha}(x) &=\sum_{m=0}^{n}(-1)^m\frac{(\alpha+1)_{n}}{m!(n-m)!(\alpha+1)_{m}}x^{m} \label{eq:Laguerre} \\
	J_{n}^{(\alpha, \beta)}(x) &= (\alpha+1)_{n} \sum_{m=0}^{n} \frac{(-1)^m(\alpha+\beta+n+1)_{m}}{m!(n-m)!(\alpha+1)_{m}}\Big( \frac{1-x}{2} \Big)^{m} \label{eq:Jacobi}
\end{align}
where $\lfloor x \rfloor$ is the floor function and $(a)_{n}=a(a+1)(a+2)\cdots(a+n-1)$ is the Pochhammer symbol. $(a)_{n}$ is related to the Gamma function as
\begin{equation*}
	(a)_n=\frac{\Gamma(a+1)}{\Gamma(a)}\frac{\Gamma(a+2)}{\Gamma(a+1)}\cdots\frac{\Gamma(a+n)}{\Gamma(a+n-1)}=\frac{\Gamma(a+n)}{\Gamma(a)}
\end{equation*}
It follows that 
\begin{equation*}
	(a)_n(a+n)_m=(a)_{n+m}
\end{equation*}

\section{Gram-Schmidt Orthogonalization} \label{sec:gramschmidt}

Gram-Schmidt process is a generic procedure for orthogonalizing a set of vectors in an inner product space. The process is described recursively as:
\begin{equation}\label{eq:gsrecursive}
	u_k=v_k-\sum_{i=0}^{k-1}\frac{\inner{v_k}{u_i}}{\inner{u_i}{u_i}}u_i
\end{equation}
where $\{v_0, v_1, \cdots, v_n\}$ is a set of linearly independent vectors in an inner product space $(V, \langle \rangle)$, and $\{u_0, u_1, \cdots, u_n\}$ is the orthogonal set produced.

The recursive formula \eqref{eq:gsrecursive} can be readily used to calculate first few orthogonal polynomials. However the calculation quickly becomes cumbersome with increasing degree of polynomial (p.645 \cite{kn:AW}).

An alternative formulation for Gram-Schmidt process is a non-recursive formula (\cite{kn:Bateman},\cite{kn:Gantmacher}) that expresses $u_k$ directly in terms of $\{v_i\}_{i=0,1\cdots,k}$. The formula is given using determinant of a Gram matrix as the following 

\begin{equation}\label{eq:gsdeterminant}
	u_k = \frac{1}{d_{k,k}} \left | 
	\begin{array}{cccc}
		\inner{v_0}{v_0} & \inner{v_1}{v_0} & \cdots & \inner{v_k}{v_0} \\
		\inner{v_0}{v_1} & \inner{v_1}{v_1} & \cdots & \inner{v_k}{v_1} \\
		\vdots & \vdots & \ddots & \vdots \\
		\inner{v_0}{v_{k-1}} & \inner{v_1}{v_{k-1}} & \cdots & \inner{v_k}{v_{k-1}} \\
		v_0 & v_1 & \cdots & v_k\\				
	\end{array}
	\right |	
\end{equation}
with $d_{0,0}=1$ and $d_{k, k}$ is the cofactor of $v_k$ so that the coefficient of $v_k$ is 1. 

The determinant formula \eqref{eq:gsdeterminant} can be proved by taking the inner product of $u_k$ with each $v_i, i=0,1\cdots,k-1$. Each inner product is 0 since it is the determinant of a matrix with two identical rows. Denote the cofactor of $v_i$ as $d_{k,i}$, the determinant formula \eqref{eq:gsdeterminant}  can be rewritten as
\begin{equation}\label{eq:gscofactor}
	u_k = \sum_{i=0}^{k}(-1)^{i+k}\frac{d_{k,i}}{d_{k,k}} v_{i} 
\end{equation}  
Each cofactor $d_{k,i}$ is the determinant of the Gram matrix defined by pairwise inner products between $\{v_0, v_1, \cdots, \widehat{v_i} \cdots,v_k\}$ and $\{v_0, v_1, \cdots, v_{k-1}\}$. A geometric interpretation of the determinant formula \eqref{eq:gsdeterminant} in terms of exterior algebra can be found in Appendix \ref{apendix:geometric-intepretation}.

Note that $u_k$ is invariant under permutation among $\{v_0, v_1, \cdots, v_{k-1}\}$. This can be seen either using formula \eqref{eq:gsrecursive} where the summation is invarint under permutation or noting that same amount of row and column operations occur in the Gram matrix under permutation. This observation will be used to simplify the calculation for Hermite polynomials in Section \ref{sec:hermite}.

Another thing to note is that orthogonal basis is unique up to a scaling factor. In Gram-Schmidt process, the coefficient for $v_k$ is one. In applications, a different scaling factor may be used for other consideration. In later sections we'll see that the orthogonal functions obtained using Gram-Schmidt process will differ from standard expressions by certain scaling factors.  

\begin{example}
	As an example, we compute the inner product between two monomials $x^i$ and $x^j$ under Laguerre weight function:
	\begin{align} \label{eq:gshermite}
		\inner{x^i}{x^j} & = \int_{0}^{+\infty} x^i x^j x^{\alpha} e^{-x} dx = \int_{0}^{+\infty} x^{i+j+\alpha} e^{-x} dx \nonumber	\\
		& = \Gamma(i+j+\alpha+1)
	\end{align}  
	Therefore the Gram matrix for monomial basis are entries defined by Gamma functions. The determinant of such matrix turns out to be related to Vandemonde determinant which is described in the next section.
\end{example} 

\section{Vandermonde Determinants} \label{sec:vandermonde}

For a number set $\mathbf{z}=(z_0,z_1,\cdots, z_{n-1})$, one can define a Vandermonde matrix with entries $z_{j-1}^{i-1}$. Its determinant has the following formula:
\begin{equation} \label{eq:vandermonde}
	\Delta_n(\mathbf{z})=\Delta_n (z_0,z_1,\cdots, z_{n-1}):=\prod_{0\le i<j \le n-1}(z_j-z_i)
\end{equation}

The Vandermonde formula \eqref{eq:vandermonde} can be proved recursively by subtracting each row by the previous row multiplied by $z_0$ and performing factorization which gives
\begin{equation*}
	\text{det} [(z_j)_i]_{i,j=0,1,\cdots,n-1}=\Big\{ \prod_{1\le i \le n-1}(z_i-z_0) \Big\} \text{det} [(z_{j+1})_i]_{i,j=0,1,\cdots,n-2}	
\end{equation*} 

A useful observation is that $\Delta_n(\mathbf{z})$ is invariant under a parallel shift, i.e.
\begin{equation*}
	\Delta_n(\mathbf{z} \oplus c):=\Delta_n (z_0+c,z_1+c,\cdots, z_{n-1}+c)=\Delta_n(\mathbf{z})
\end{equation*}

For $m=0,1,\cdots,n$, define set $\mathbf{e}_n^m:= (1, 2, \cdots, \widehat{m+1},\cdots, n+1)$. In particular $\mathbf{e}_n^n= (1, 2, \cdots, n)$. Straightforward calculation shows
\begin{proposition}
	\begin{equation}\label{ed:intdelta}
		\frac{\Delta_n(\mathbf{e}_n^m)}{\Delta_n(\mathbf{e}_n^n)}=\frac{\Delta_n(1, 2, \cdots, \widehat{m+1},\cdots, n+1)}{\Delta_n(1, 2, \cdots, n)}=\frac{n!}{m!(n-m)!}	
	\end{equation} 
\end{proposition}

The following are the generalized Vandermonde determinant formula for matrices with entries defined by Gamma and Beta functions.
\begin{proposition}
	\begin{align}
		\textup{det} [\Gamma(z_j+i)]_{n \times n}& =\Big\{ \prod_{j=0}^{n-1} \Gamma(z_j) \Big\} \Delta_n(\mathbf{z}) \label{eq:detgamma}	\\
		\textup{det} [B(z_j+i,w)]_{n \times n}&=\Big\{\prod_{j=0}^{n-1} \frac{\Gamma(z_j) \Gamma(w+j)}{\Gamma(z_j+w+n-1)}\Big\} \Delta_n(\mathbf{z}) \label{eq:detbeta}	
	\end{align}	
\end{proposition}

Formula \eqref{eq:detgamma} can be derived by showing that
\begin{equation*}
	\text{det} [(z_j)_i]_{n \times n}=\Delta_n(\mathbf{z})
\end{equation*}
which is reduced to Vandermonde determinant by matrix row operations since $(z_j)_i$ is a degree $i$ polynomial in $z_j$.

Formula \eqref{eq:detbeta} is not exactly stated in \cite{kn:Normand}, but is straightforward from Corollary 6 in \cite{kn:Normand}. It is derived using the following result (Lemma 3 in \cite{kn:Normand})
\begin{equation} \label{eq:detbeta0}
	\text{det} \Big [\frac{(z_j)_i}{(z_j+w)_i} \Big ]_{n \times n}=\Big \{\prod_{j=0}^{n-1} \frac{(w)_j}{(z_j+w)_{n-1}}\Big \} \Delta_n(\mathbf{z})
\end{equation}
which can be derived recursively in a similar spirit as Vandermonde determinant. Indeed if we subtract $(i+1)$-th row by $i$-th row multiplied by a factor to zero out the first entry in the row, the $(i+1, j+1)$-th entry becomes:
\begin{align*}
	&\frac{(z_j)_i}{(z_j+w)_i}-\frac{z_0+i-1}{z_0+w+i-1} \frac{(z_j)_{i-1}}{(z_j+w)_{i-1}} \\
	= &\frac{w(z_j-z_0)}{(z_0+w+i-1)(z_j+w+i-1)}\frac{(z_j)_{i-1}}{(z_j+w)_{i-1}} \\
	= &\frac{w(z_j-z_0)}{(z_0+w+i-1)(z_j+w)}\frac{(z_j)_{i-1}}{(z_j+w+1)_{i-1}}
\end{align*}
Therefore we obtain a recursive formula for the determinant
\begin{align*}
	&\text{det} \Big [\frac{(z_j)_i}{(z_j+w)_i} \Big ]_{n \times n} \\
	=&\Big \{\prod_{i=1}^{n-1} \frac{w(z_i-z_0)}{(z_i+w)(z_0+w+i-1)}\Big \}\text{det} \Big [\frac{(z_{j+1})_i}{(z_{j+1}+w+1)_i} \Big ]_{(n-1) \times (n-1)}	\\
	=&\frac{w^{n-1}}{(z_0+w)_{n-1}}\Big \{\prod_{j=1}^{n-1} \frac{(z_j-z_0)}{(z_j+w)}\Big \}\text{det} \Big [\frac{(z_{j+1})_i}{(z_{j+1}+w+1)_i} \Big ]_{(n-1) \times (n-1)}
\end{align*}
Note that $w$ is increased by 1 in the next iteration. Recursively we have
\begin{align*}
	\text{det} \Big [\frac{(z_j)_i}{(z_j+w)_i} \Big ]_{n \times n}
	=&\prod_{i=1}^{n-1} \frac{(w+i-1)^{n-i}}{(z_{i-1}+w+i-1)_{n-i}} \Big \{ \prod_{j=i}^{n-1} \frac{(z_j-z_{i-1})}{(z_j+w+i-1)}\Big \}	
\end{align*}
We split the right hand side into four terms. The first term is
\begin{align*}
	\prod_{i=1}^{n-1}(w+i-1)^{n-i}&=	\prod_{i=1}^{n-1}\prod_{j=i}^{n-1}(w+i-1)=\prod_{j=1}^{n-1}\prod_{i=1}^{j}(w+i-1)=\prod_{j=1}^{n-1}(w)_{j}
\end{align*}
The second terms is 
\begin{align*}
	\prod_{i=1}^{n-1} \frac{1}{(z_{i-1}+w+i-1)_{n-i}} &=\prod_{i=1}^{n-1} \frac{(z_{i-1}+w)_{i-1}}{(z_{i-1}+w)_{n-1}} =\prod_{i=0}^{n-2} \frac{(z_{i}+w)_{i}}{(z_{i}+w)_{n-1}} 	
\end{align*}
The third term is
\begin{align*}
	\prod_{i=1}^{n-1} \prod_{j=i}^{n-1} \frac{1}{(z_j+w+i-1)} &= \prod_{i=1}^{n-1} \prod_{j=1}^{i} \frac{1}{(z_i+w+j-1)}= \prod_{i=1}^{n-1} \frac{1}{(z_i+w)_i}
\end{align*}
The fourth term is
\begin{align*}
	\prod_{i=1}^{n-1} \prod_{j=i}^{n-1} (z_j-z_{i-1}) = \prod_{0\le i<j \le n-1}(z_j-z_i)=\Delta_n(\mathbf{z})
\end{align*}
Combining all four terms we have
\begin{align*}
	\text{det} \Big [\frac{(z_j)_i}{(z_j+w)_i} \Big ]_{n \times n}
	=& \Big \{\prod_{j=1}^{n-1}(w)_{j} \} \Big \{\prod_{i=0}^{n-2} \frac{(z_{i}+w)_{i}}{(z_{i}+w)_{n-1}} \Big \} \Big \{ \prod_{i=1}^{n-1} \frac{1}{(z_i+w)_i}\Big \} \Delta_n(\mathbf{z})\\ 
	=& \Big \{\prod_{j=0}^{n-1} \frac{(w)_j}{(z_j+w)_{n-1}}\Big \} \Delta_n(\mathbf{z})		
\end{align*}

\section{Laguerre Polynomials} \label{sec:laguerre}

Here we derive Laguerre polynomials using the determinant formulation \eqref{eq:gscofactor} with the monomials $\{1, x, x^2, \cdots, \}$ as basis. As computed in formula \eqref{eq:gshermite}, the Gram matrix has entries in Gamma function. 

We first compute cofactor $d_{n,n}$ by setting $\mathbf{z}=\mathbf{e}_n^n\oplus\alpha$ in formula \eqref{eq:detgamma}
\begin{align*}
	d_{n,n} & = \text{det} [\Gamma(z_j+i)]_{i,j=0,1,\cdots,n-1}=\Big\{ \prod_{j=0}^{n-1} \Gamma(z_j) \Big\} \Delta_n(\mathbf{z})\\	
	& = \Big\{ \prod_{j=0}^{n-1} \Gamma(\alpha+1+j) \Big\} \Delta_n(\mathbf{e}_n^n)  		 
\end{align*}

We then compute cofactor $d_{n,m}$ by set $\mathbf{z}=\mathbf{e}_n^m\oplus\alpha$ in formula \eqref{eq:detgamma} 
\begin{align*}
	d_{n,m} &= \text{det} [\Gamma(z_j+i)]_{i,j=0,1,\cdots,n-1}=\Big\{ \prod_{j=0}^{n-1} \Gamma(z_j) \Big\} \Delta_n(\mathbf{z})\\	
	& = \Big\{ \prod_{j=0, j \neq m}^{n}\Gamma(\alpha+1+j) \Big\} \Delta_n(\mathbf{e}_n^m)		 
\end{align*}

Therefore the $n$-th orthogonal function using formula \eqref{eq:gscofactor} is 
\begin{align*}
	l_n^{\alpha}(x) & =  \sum_{m=0}^{n}(-1)^{n+m}\frac{\Big\{ \prod_{j=0, j \neq m}^{n}\Gamma(\alpha+1+j) \Big\} \Delta_n(\mathbf{e}_n^m)}{\Big\{ \prod_{j=0}^{n-1} \Gamma(\alpha+1+j) \Big\} \Delta_n(\mathbf{e}_n^n)} x^m \\
	& = \sum_{m=0}^{n}(-1)^{n+m}\frac{\Gamma(\alpha+n+1) n!}{\Gamma(\alpha+m+1) m!(n-m)!} x^m \\
	& = \sum_{m=0}^{n}(-1)^{n+m}\frac{n!(\alpha+1)_n}{m!(n-m)!(\alpha+1)_m} x^m  	 		 
\end{align*}

Comparing this to the formula \eqref{eq:Laguerre} we conclude
\begin{equation}
	l_n^{\alpha}(x) = (-1)^{n}n!L_n^{\alpha}(x)			 
\end{equation}

\section{Jacobi Polynomials} \label{sec:jacobi}

Here we derive Jacobi polynomials using the determinant formulation \eqref{eq:gscofactor} with the set $\{1, \frac{1-x}{2}, (\frac{1-x}{2})^2, \cdots, \}$ as basis. This choice is an adaption to the form of Jacobi weight function.

We first compute the inner product between $f_i=(\frac{1-x}{2})^i$ and $f_j=(\frac{1-x}{2})^j$ under Jacobi weight function:
\begin{align*}
	\inner{f_i}{f_j} & = \int_{-1}^{1} (\frac{1-x}{2})^i (\frac{1-x}{2})^i  (1-x)^{\alpha} (1+x)^{\beta}dx \\
	& =2^{\alpha+\beta+1} \int_{0}^{1} u^{i+j+\alpha} (1-u)^{\beta} du	\\
	& =2^{\alpha+\beta+1} B(i+j+\alpha+1, \beta+1)
\end{align*}

For cofactor $d_{n,n}$, we set $\mathbf{z}=\mathbf{e}_n^n\oplus\alpha$ and $w=\beta+1$ in formula \eqref{eq:detbeta}
\begin{align*}
	d_{n,n} & = \text{det} [2^{\alpha+\beta+1}B(z_j+i,w)]_{n \times n}=\Big\{\prod_{j=0}^{n-1} \frac{2^{\alpha+\beta+1}\Gamma(z_j) \Gamma(w+j)}{\Gamma(z_j+w+n-1)}\Big\} \Delta_n(\mathbf{z}) \\	
	& =\Big \{\prod_{j=0}^{n-1} \frac{2^{\alpha+\beta+1}\Gamma(\alpha+1+j)\Gamma(\beta+1+j)}{\Gamma(\alpha+\beta+n+1+j)}\Big \} \Delta_n (\mathbf{e}_n^n)			 
\end{align*}

For cofactor $d_{n,m}$, we set $\mathbf{z}=\mathbf{e}_n^m\oplus\alpha$ and $w=\beta+1$ in formula \eqref{eq:detbeta}
\begin{align*}
	d_{n,m}  & = \text{det} [2^{\alpha+\beta+1}B(z_j+i,w)]_{n \times n}=\Big\{\prod_{j=0}^{n-1} \frac{2^{\alpha+\beta+1}\Gamma(z_j) \Gamma(w+j)}{\Gamma(z_j+w+n-1)}\Big\} \Delta_n(\mathbf{z}) \\
	& = \Big \{\prod_{j=0, j \neq m}^{n} \frac{2^{\alpha+\beta+1}\Gamma(\alpha+1+j)}{\Gamma(\alpha+\beta+n+1+j)}\prod_{j=0}^{n-1} \Gamma(\beta+1+j)\Big \} \Delta_n(\mathbf{e}_n^m) 		 		 
\end{align*}

Therefore the $n$-th orthogonal function using formula \eqref{eq:gscofactor} is 
\begin{align*}
	j_n^{\alpha,\beta}(x) & = \sum_{m=0}^{n}(-1)^{n+m}\frac{\Big \{\prod_{j=0, j \neq m}^{n} \frac{\Gamma(\alpha+1+j)}{\Gamma(\alpha+\beta+n+1+j)}\Big \} \Delta_n(\mathbf{e}_n^m)}{\Big \{\prod_{j=0}^{n-1} \frac{\Gamma(\beta+1+j)}{\Gamma(\alpha+\beta+n+1+j)}\Big \} \Delta_n(\mathbf{e}_n^n)} (\frac{1-x}{2})^m \\
	& = \sum_{m=0}^{n}(-1)^{n+m}\frac{\Gamma(\alpha+n+1)\Gamma(\alpha+\beta+n+m+1)n!} {\Gamma(\alpha+\beta+2n+1)\Gamma(\beta+m+1)m!(n-m)!} (\frac{1-x}{2})^m \\	  
	& = \frac{(-1)^{n}n!(\alpha+1)_{n}}{(\alpha+\beta+n+1)_{n}}\sum_{m=0}^{n}\frac{(-1)^{m}(\alpha+\beta+n+1)_m} {m!(n-m)!(\alpha+1)_{m} } (\frac{1-x}{2})^m 				 
\end{align*}

Comparing this to the formula \eqref{eq:Jacobi} we conclude
\begin{equation}
	j_n^{\alpha,\beta}(x)=\frac{(-1)^{n}n!}{(\alpha+\beta+n+1)_{n}} J_{n}^{\alpha,\beta}(x)			 
\end{equation}

\section{Hermite Polynomials} \label{sec:hermite}

Here we derive Hermite polynomials using the determinant formulation \eqref{eq:gscofactor} with the monomials $\{1, x, x^2, \cdots, \}$ as basis. 

We first compute the inner product between two monomials $x^i$ and $x^j$ under Hermite weight function:
\begin{align*}
	\inner{x^i}{x^j} & = \int_{-\infty}^{+\infty} x^i x^j e^{-x^2} dx = \int_{-\infty}^{+\infty} x^{i+j} e^{-x^2} dx	\\
	& = \begin{cases}
		0        & \quad \text{if } i+j \text{ is odd} \\
		\int_{0}^{+\infty} x^{\frac{i+j-1}{2}} e^{-x} dx & \quad \text{if } i+j \text{ is even} \\		
	\end{cases} \\
	& = \begin{cases}
		0        & \quad \text{if } i+j \text{ is odd} \\
		\Gamma(\frac{i+j}{2}+\frac{1}{2}) & \quad \text{if } i+j \text{ is even} \\		
	\end{cases}
\end{align*}

For the even $2n$ case, we rearrange the monomial basis up to $x^{2n}$ by first odd degree and then even degree, i.e. $(x, x^3, \cdots, x^{2n-1}, 1, x^2, \cdots, x^{2n})$. Since the inner product between even and odd degree monomials are 0, the $2n \times (2n+1)$ Gram matrix has a form with non zero entries at the upper left $ n \times n $ block and lower right $n \times (n+1)$ block, it is therefore only the even degree mononial has non-zero cofactor. In addition the $n \times (n+1)$ block has the form $[\Gamma(i+j+\frac{1}{2})]_{i=0,\cdots,n-1,j=0,\cdots,n}$ which is the one in the Laguerre polynomials with $\alpha=-1/2$. As the upper left matrix appears in all cofactors, the cofactor ratio is the same as that in Laguerre case which is given as
\begin{align*}
	\frac{d_{2n,n+m}}{d_{2n,2n}} &=\frac{n!(\frac{1}{2})_n}{m!(n-m)!(\frac{1}{2})_m}=\frac{n!(2n-1)!!}{2^{n-m}m!(n-m)!(2m-1)!!} \\
	&=	\frac{(2n)!}{2^{2(n-m)}(n-m)!(2m)!}
\end{align*}

Therefore the $2n$-th orthogonal function using formula \eqref{eq:gscofactor} is
\begin{align*}
	h_{2n}(x)& = \sum_{m=0}^{n}(-1)^{2n+n+m}\frac{(2n)!}{2^{2(n-m)}(n-m)!(2m)!}x^{2m} \\
	& = \sum_{m=0}^{n}(-1)^{n+m}\frac{(2n)!}{2^{2n}(n-m)!(2m)!}(2x)^{2m} \\
	& = \sum_{m=0}^{n}(-1)^{m}\frac{(2n)!}{2^{2n}(m)!(2n-2m)!}(2x)^{2n-2m}		 
\end{align*}

For the odd $2n+1$ case, we rearrange the monomial basis up to $x^{2n+1}$ by first even degree and then odd degree, i.e. $(1, x^2, \cdots, x^{2n}, x, x^3, \cdots, x^{2n+1})$. The $(2n+1) \times (2n+2)$ Gram matrix has a form with non zero entries at the upper left $ (n+1) \times (n+1) $ block and lower right $n \times (n+1)$ block, it is therefore only the odd degree mononial has non-zero cofactor. In addition the $n \times (n+1)$ block has the form $[\Gamma(i+j+\frac{3}{2})]_{i=0,\cdots,n-1,j=0,\cdots,n}$ which is the one in the Laguerre polynomials with $\alpha=1/2$. Therefore the cofactor ratio is the same as that in Laguerre case which is given as
\begin{align*}
	\frac{d_{2n+1,n+1+m}}{d_{2n+1,2n+1}} &=\frac{n!(\frac{3}{2})_n}{m!(n-m)!(\frac{3}{2})_m}=\frac{n!(2n+1)!!}{2^{n-m}m!(n-m)!(2m+1)!!} \\
	&=	\frac{(2n+1)!}{2^{2(n-m)}(n-m)!(2m+1)!}
\end{align*}

Therefore the $(2n+1)$-th orthogonal function using formula \eqref{eq:gscofactor} is
\begin{align*}
	h_{2n+1}(x)& = \sum_{m=0}^{n}(-1)^{2n+1+n+1+m}\frac{(2n+1)!}{2^{2(n-m)}(n-m)!(2m+1)!}x^{2m+1} \\
	& = \sum_{m=0}^{n}(-1)^{n+m}\frac{(2n+1)!}{2^{2n+1}(n-m)!(2m+1)!}(2x)^{2m+1} \\
	& = \sum_{m=0}^{n}(-1)^{m}\frac{(2n+1)!}{2^{2n+1}(m)!(2n+1-2m)!}(2x)^{2n+1-2m}		 
\end{align*}

Combining the even and odd cases together and comparing with the formula \eqref{eq:Hermite}, we conclude
\begin{equation}
	h_{n}(x)=\sum_{m=0}^{\lfloor \frac{n}{2} \rfloor}\frac{(-1)^{m}n!}{2^n(m)!(n-2m)!}(2x)^{n-2m}=\frac{n!}{2^n}H_n(x)
\end{equation}


\addresseshere

\newpage

\section{Appendix: Geometric Interpretation of the Determinant Formula for Gram-Schidmit Process}\label{apendix:geometric-intepretation}

Here we give a geometric interpretation of the determinant formula \eqref{eq:gsdeterminant} using exterior algebra. Indeed if we set $V_{k}:=\text{Span}\{v_0, v_1,\cdots, v_{k}\}$, geometrically Gram-Schmidt orthogonalization is to find the (one-dimensional) orthogonal compliment of $V_{k-1}$ in $V_k$. The recursive formula \ref{eq:gsrecursive} is simply obtained by orthogonal projection using orthogonal basis. The determinant formula however expresses the orthogonal compliment using general non-orthogonal basis which is to be seen related to the Hodge star operator.   

Let's first recall some basic properties of the Hodge star operator from Warner's book (\cite{kn:W}). Let $(W, \langle \rangle)$ be an $m$-dimensional oriented vector space with an inner product that naturally extends to its exterior algebra $\wedge W$ by setting 
\begin{equation}\label{eq:exterior-algebra-inner-product}
	\inner{w_1 \wedge w_2 \wedge \cdots \wedge w_p}{w_1^{\prime} \wedge w_2^{\prime} \wedge \cdots \wedge w_p^{\prime}} =\text{det}(\inner{w_i}{w_j^{\prime}}_{1\le i, j \le p})	
\end{equation}
on $p$-form. The Hodge star operator $\ast: \wedge^{p}W \to \wedge^{m-p} W$ is determined by
\[ \inner{\ast \alpha}{\beta} \, \Omega = \alpha \wedge \beta, \, \forall \, \alpha\in \wedge^{p}W, \,\beta \in \wedge^{m-p}W \]
where $\Omega \in \wedge^{m}W$ is the volume form associated with the inner product and orientation. The Hodge star operator has the following properties:

\begin{enumerate}
	\item $\ast 1 = \Omega$
	\item $\ast \ast \alpha = (-1)^{p(m-p)}, \, \forall \, \alpha\in \wedge^{p}W $	
	\item $ \inner{\ast \alpha}{\ast \beta} = \inner{\alpha}{\beta}, \, \forall \, \alpha, \beta \in \wedge^{p}W $
	\item $ \inner{\alpha}{\beta} = \ast(\beta \wedge \ast \alpha) = \ast(\alpha \wedge \ast \beta), \, \forall \, \alpha, \beta \in \wedge^{p}W$
\end{enumerate}

Back to the set up in the Gram-Schmidt process, for $1\le k \le n$, we define $u_k^{\prime} \in V_k$ using the Hodge star operator
\begin{equation}\label{eq:gram-schmidt-hodge-star}
	u_k^{\prime}:=\ast(v_0 \wedge v_1 \wedge \cdots \wedge v_{k-1})	
\end{equation}

\begin{lemma}
	\[ \inner{u_k^{\prime}}{v_i} = 0, 0\le i <k \le n\]
\end{lemma}   

\begin{proof}
	Indeed by the properties of the Hodge star operator we have 
	\[ \inner{u_k^{\prime}}{v_i}= \inner{\ast(v_0 \wedge v_1 \wedge \cdots \wedge v_{k-1})}{v_i} = \ast(v_0 \wedge v_1 \wedge \cdots \wedge v_{k-1} \wedge v_i) \]
	which is 0 for $0\le i < k$ since $v_i$ appears twice in the wedge product.
\end{proof} 

From the lemma we have $u_k^{\prime}$ in the orthogonal compliment of $V_{k-1}$ in $V_k$, thus is proportional to $u_k$. To find the scaling factor, we notice that
\begin{align*}
	\inner{u_k^{\prime}}{u_k^{\prime}} & = \inner{\ast(v_0 \wedge v_1 \wedge \cdots \wedge v_{k-1})}{\ast(v_0 \wedge v_1 \wedge \cdots \wedge v_{k-1})} \\ 
	& =\text{det}(\inner{v_i}{v_j}_{0\le i, j \le k-1})=d_{k,k} \\
	\inner{u_k^{\prime}}{v_k} & = \ast(v_0 \wedge v_1 \wedge \cdots \wedge v_{k})\\
	&=\sqrt{\text{det}(\inner{v_i}{v_j}_{1\le i, j \le k})} =\sqrt{d_{k+1,k+1}}
\end{align*}

Therefore 
\begin{equation}\label{eq:gshodgestar}
u_k=\frac{\sqrt{d_{k+1,k+1}}}{d_{k,k}}u_k^{\prime}=\frac{\sqrt{d_{k+1,k+1}}}{d_{k,k}}\ast(v_0 \wedge v_1 \wedge \cdots \wedge v_{k-1}) 	
\end{equation}

To see how the equation \eqref{eq:gshodgestar} gives the right side of the determinant formula \eqref{eq:gsdeterminant} directly, we consider the inner product of $u_k^{\prime}$ with any vector $v$

\begin{align*}
	\inner{u_k^{\prime}}{v} &= \ast(v_0 \wedge v_1 \wedge \cdots \wedge v_{k-1} \wedge v ) \\
	&=\frac{\inner{\ast(v_0 \wedge v_1 \wedge \cdots \wedge v_{k-1}\wedge v)}{\ast(v_0 \wedge v_1 \wedge \cdots \wedge v_{k-1}\wedge v_k)}}{\ast(v_1 \wedge v_2 \wedge \cdots \wedge v_{k-1} \wedge v_k)} 	\\
	&=\frac{\inner{v_0 \wedge v_1 \wedge \cdots \wedge v_{k-1}\wedge v}{v_0 \wedge v_1 \wedge \cdots \wedge v_{k-1}\wedge v_k}}{\sqrt{d_{k+1,k+1}}}  \\
	& = \frac{1}{\sqrt{d_{k+1,k+1}}} \left | 
	\begin{array}{cccc}
		\inner{v_0}{v_0} & \inner{v_1}{v_0} & \cdots & \inner{v_k}{v_0} \\
		\inner{v_0}{v_1} & \inner{v_1}{v_1} & \cdots & \inner{v_k}{v_1} \\
		\vdots & \vdots & \ddots & \vdots \\
		\inner{v_0}{v_{k-1}} & \inner{v_1}{v_{k-1}} & \cdots & \inner{v_k}{v_{k-1}} \\
		\inner{v_0}{v} & \inner{v_1}{v} & \cdots & \inner{v_k}{v}\\				
	\end{array}
	\right |	
\end{align*}  
 
Therefore 
\begin{equation*}
u_k^{\prime}=\frac{1}{\sqrt{d_{k+1,k+1}}}\left | 
\begin{array}{cccc}
	\inner{v_0}{v_0} & \inner{v_1}{v_0} & \cdots & \inner{v_k}{v_0} \\
	\inner{v_0}{v_1} & \inner{v_1}{v_1} & \cdots & \inner{v_k}{v_1} \\
	\vdots & \vdots & \ddots & \vdots \\
	\inner{v_0}{v_{k-1}} & \inner{v_1}{v_{k-1}} & \cdots & \inner{v_k}{v_{k-1}} \\
	v_0& v_1 & \cdots & v_k \\				
\end{array}
\right |	
\end{equation*}

Now we can see the right side of the Equation \eqref{eq:gshodgestar} is the same as the right side of the Equation \eqref{eq:gsdeterminant}. This completes the geometric interpretation of the determinant formula Equation \eqref{eq:gsdeterminant} in terms of the Hodge star operator in exterior algebra.


\begin{thebibliography}{Xyz}
	
\bibitem{kn:AW} G.B. Arfken and H.J. Weber, {\em Mathematical Methods for Physicists}, 6th Edition, Elsevier Academic Press 2005.

\bibitem{kn:Bateman} H.Bateman, {\em Higher Transcendental Functions} vol 2, McGraw-Hill 1953.

\bibitem{kn:BealsWong} R. Beals and R. Wong,{\em Special Functions and Orthogonal Polynomials},  Cambridge Studies in Advanced Mathematics 153, Cambridge University Press 2016.

\bibitem{kn:CH} R. Courant and D. Hilbert, {\em Methods of Mathematical Physics}, 1st English Edition, Interscience Publishers, INC., New York 1937.

\bibitem{kn:Gantmacher} F. Gantmacher, {\em Theory of Matrices},  AMS Chelsea Publishing, 1959.

\bibitem{kn:Normand} J.M. Normand,{\em Calculation of some determinants using the s-shifted factorial}, Journal of Physics A: Mathematical and General, Volume 37, Number 22, IOP Publishing Ltd 2004.

\bibitem{kn:W} F.W. Warner, {\em Foundations of Differentiable Manifolds and Lie Groups},  Graduate Texts in Mathematics (Volume {\bf 94}), Springer 1983.

\end{thebibliography}
\end{document}